\documentclass[notitlepage,11pt,reqno]{amsart}
\usepackage{amssymb,amsmath,amsthm,soul,mathrsfs,color}
\usepackage[breaklinks=true]{hyperref}

\newcommand{\bb}{\ensuremath{\mathbb B}}
\newcommand{\be}{\ensuremath{\mathbb E}}
\newcommand{\N}{\ensuremath{\mathbb N}}
\newcommand{\G}{\ensuremath{\mathscr G}}
\newcommand{\bp}[1]{\ensuremath{\mathbb P}_\mu \left( #1 \right)}
\newcommand{\conv}[1]{\ensuremath{\stackrel{#1}{\longrightarrow}} }

\newtheorem{theorem}[equation]{Theorem}
\newtheorem{prop}[equation]{Proposition}
\newtheorem{lemma}[equation]{Lemma}

\newtheorem{utheorem}{\textrm{\textbf{Theorem}}}

\theoremstyle{definition}
\newtheorem{remark}[equation]{Remark}
\newtheorem{definition}[equation]{Definition}

\numberwithin{equation}{section}

\usepackage[letterpaper]{geometry}
\begin{document}
\title[Probability inequalities and tail estimates for metric
semigroups]{Probability inequalities and tail estimates\\ for metric
semigroups}

\author{Apoorva Khare}
\address[A.~Khare]{Department of Mathematics, Indian Institute of
Science; Analysis and Probability Research Group; Bangalore, India}
\email{khare@iisc.ac.in}

\author{Bala Rajaratnam}
\address[B.~Rajaratnam]{University of California, Davis, USA and
University of Sydney, Australia (visiting)}
\email{brajaratnam01@gmail.com}

\thanks{This work was partially supported by the following: US Air Force
Office of Scientific Research grant award FA9550-13-1-0043, US National
Science Foundation under grant DMS-0906392, DMS-CMG 1025465, AGS-1003823,
DMS-1106642, DMS-CAREER-1352656, Defense Advanced Research Projects
Agency DARPA YFA N66001-111-4131, the UPS Foundation, SMC-DBNKY, 
Ramanujan Fellowship SB/S2/RJN-121/2017 and MATRICS grant MTR/2017/000295
from SERB (Govt.~of India), by grant F.510/25/CAS-II/2018(SAP-I) from UGC
(Govt.~of India), and a Young Investigator Award from the Infosys
Foundation.}



\date{\today}

\begin{abstract}
We study probability inequalities leading to tail estimates in a general
semigroup $\G$ with a translation-invariant metric $d_\G$. (An important
and central example of this in the functional analysis literature is that
of $\G$ a Banach space.)
Using our prior work [{\it Ann.~Prob.} 2017] that extends the
Hoffmann-J{\o}rgensen inequality to all metric semigroups, we obtain tail
estimates and approximate bounds for sums of independent semigroup-valued
random variables, their moments, and decreasing rearrangements. In
particular, we obtain the ``correct'' universal constants in several
cases, extending results in the Banach space literature by
Johnson--Schechtman--Zinn [{\it Ann.~Prob.} 1985], Hitczenko [{\it
Ann.~Prob.} 1994], and Hitczenko and Montgomery-Smith [{\it Ann.~Prob.}
2001]. Our results also hold more generally, in a very primitive
mathematical framework required to state them: metric semigroups $\G$.
This includes all compact, discrete, or (connected) abelian Lie groups.
\end{abstract}
\maketitle

\section{Introduction and main results}

This paper follows our prior work~\cite{KR3} and continues the study of
probability theory beyond -- but also subsuming -- the Banach space
setting. In the present work, we estimate sums of independent random
variables in several ways, under very primitive mathematical assumptions
that suffice to state our results. The setting is as follows.

\begin{definition}\label{Dsemi}
A \textit{metric semigroup} is defined to be a semigroup $(\G, \cdot)$
equipped with a metric $d_\G : \G \times \G \to [0,\infty)$ that is
translation-invariant. In other words,
\[ d_\G(ac,bc) = d_\G(a,b) = d_\G(ca,cb)\ \forall a,b,c \in \G. \]

\noindent (Equivalently, $(\G,d_\G)$ is a metric space equipped with a
associative binary operation such that $d_\G$ is translation-invariant.)
Similarly, one defines a \textit{metric monoid} and a \textit{metric
group}.
\end{definition}

Metric groups are ubiquitous in probability theory and functional
analysis, and subsume all normed linear spaces as well as compact and
(connected) abelian Lie groups as special cases. More modern examples of
recent interest are mentioned presently.

Now suppose $(\Omega, \mathscr{A}, \mu)$ is a probability space and $X_1,
\dots, X_n \in L^0(\Omega,\G)$ are $\G$-valued random variables. Fix
$z_0, z_1 \in \G$ and define for $1 \leqslant j \leqslant n$:
\begin{equation}\label{Esemigroup}
S_j := X_1 X_2 \cdots X_j, \quad
U_j := \max_{1 \leqslant i \leqslant j} d_\G(z_1, z_0 S_i), \quad
Y_j := d_\G(z_0, z_0 X_j), \quad
M_j := \max_{1 \leqslant i \leqslant j} Y_i.
\end{equation}

In this paper we discuss bounds that govern the behavior of $U_n$ -- and
consequently, of sums $S_n$ of independent $\G$-valued random variables
$X_j$ -- in terms of the variables $X_j$, and even $Y_j$ or $M_j$.
We are interested in a variety of bounds:
(a)~one-sided geometric tail estimates;
(b)~approximate two-sided bounds for tail probabilities;
(c)~approximate two-sided bounds for moments; and
(d)~comparison of moments.
For instance, is it possible to obtain bounds for $\be_\mu[U_n^p]^{1/p}$
in terms of the tail distribution for $U_n$, or in terms of
$\be_\mu[U_n^q]^{1/q}$ for $p,q > 0$? The latter question has been
well-studied in the literature for Banach spaces, and universal bounds
that grow at the ``correct'' rate have been obtained for all $q \gg 0$.
We explore the question of obtaining correctly growing universal
constants for metric semigroups, which include not only normed linear
spaces and inner product spaces, but also all connected abelian and
compact Lie groups. Our results show that the universal constants in such
inequalities do not depend on the semigroup in question.

\subsection{Motivations}

Our motivations in developing probability theory in such general settings
are both modern and classical. An increasing number of modern-day
theoretical and applied settings require mathematical frameworks that go
beyond Banach spaces. For instance, data and random variables may take
values in manifolds such as (real or complex) Lie groups. Compact or
connected abelian Lie groups also commonly feature in the literature,
including permutation groups and other finite groups, lattices,
orthogonal groups, and tori.
In fact every abelian, Hausdorff, metrizable, topologically complete
group $G$ admits a translation-invariant metric \cite{Kl}, though this
fails to hold for cancellative semigroups \cite{KG}.
Certain classes of amenable groups are also metric groups (see \cite{KR4}
for more details).
Other modern examples arise in the study of large networks and include
the space of graphons with the cut norm, which arises naturally out of
combinatorics and is related to many applications \cite{Lo}. In a
parallel vein, the space of labelled graphs $\G(V)$ on a fixed vertex set
$V$ is a $2$-torsion metric group (see \cite{KR1,KR2}), hence does not
embed into a normed linear space.

With the above settings in mind, in this paper we develop novel
techniques for proving maximal inequalities -- as well as comparison
results between tail distributions and various moments -- for sums of
independent random variables taking values in the aforementioned groups,
which need not be Banach spaces.

At the same time, we also have theoretical motivations in mind when
developing probability theory on non-linear spaces such as $\G(V)$ and
beyond. Throughout the past century, the emphasis in probability has
shifted somewhat from proving results on stochastic convergence, to
obtaining sharper and stronger bounds on random sums, in increasingly
weaker settings. A celebrated achievement of probability theory has been
to develop a rigorous and systematic framework for studying the behavior
of sums of (independent) random variables; see e.g.~\cite{LT}. In this
vein, we provide \textit{unifications} of our results on graph space with
those in the Banach space literature, by proving them in a more primitive
mathematical framework encompassing both of these (and other) settings.
In particular, our results apply to compact/abelian/discrete Lie groups,
as well as normed linear spaces.

For example, maximal inequalities by Hoffmann-J{\o}rgensen, L\'evy,
Ottaviani--Skorohod, and Mogul'skii require merely the notions of a
metric and a binary associative operation to state them. Thus one only
needs a separable metric semigroup $\G$ rather than a Banach space to
state these inequalities. However, note that working in a metric
semigroup raises technical questions. For instance, the lack of an
identity element means one has to specify how to compute magnitudes of
$\G$-valued random variables (before trying to bound or estimate them);
also, it is not apparent how to define truncations of random variables.
The lack of inverses, norms, or commutativity implies in particular that
one cannot rescale or subtract random variables.

In the present work, we explain how to overcome these challenges. We also
hope to show that the approach of working with arbitrary metric
semigroups turns out to be richly rewarding in
(i)~obtaining the above (and other) results for non-Banach settings;
(ii)~unifying these results with the existing Banach space results in
order to hold in the greatest possible generality; and
(iii)~further strengthening these unified versions where possible.

\subsection{Organization and results}

We now describe the organization and contributions of the present paper.
In Section~\ref{Slevy} we prove the Mogul'skii--Ottaviani--Skorohod
inequalities for all metric semigroups $\G$. As an application, we show
L\'evy's equivalence for stochastic convergence in metric semigroups.

In Section~\ref{Stails}, we come to our main goal in this paper, of
estimating and comparing moments and tail probabilities for sums of
independent $\G$-valued random variables. Our main tool is a variant of
Hoffmann-J{\o}rgensen's inequality for metric semigroups, which is shown
in recent work \cite{KR3}. The relevant part for our purposes is now
stated.

\begin{theorem}[Khare and Rajaratnam, \cite{KR3}]\label{Thj}
Notation as in Definition~\ref{Dsemi} and Equation~\eqref{Esemigroup}.
Suppose $X_1, \dots$, $X_n \in L^0(\Omega,\G)$ are independent.
Fix scalars
\[
n_1, \dots, n_k \in \N, \qquad t_1, \dots, t_k, s \in
[0,\infty),
\]
and define
\[
I_0 := \{ 1 \leqslant i \leqslant k : \bp{U_n \leqslant t_i}^{n_i -
\delta_{i1}} \leqslant 1/n_i! \},
\]
where $\delta_{i1}$ denotes the Kronecker delta.
Now if $\sum_{i=1}^k n_i \leqslant n+1$, then:
\begin{align*}
&\ \bp{U_n > (2 n_1 - 1) t_1 + 2 \sum_{i=2}^k n_i t_i + \left(
\sum_{i=1}^k n_i - 1 \right) s}\\
\leqslant &\ \bp{M_n > s} + \bp{U_n \leqslant t_1}^{{\bf 1}_{1 \notin
I_0}} \prod_{i \in I_0} \bp{U_n > t_i}^{n_i} \prod_{i \notin I_0}
\frac{1}{n_i!} \left( \frac{\bp{U_n > t_i}}{\bp{U_n \leqslant t_i}}
\right)^{n_i}.
\end{align*}
\end{theorem}

Remark that Theorem~\ref{Thj} generalizes the original
Hoffmann-J{\o}rgensen inequality in three ways:
(i)~mathematically it strengthens the state-of-the-art even for real
variables;
(ii)~it unifies previous results by Johnson and Schechtman~\cite{JS},
Klass and Nowicki~\cite{KN}, and Hitczenko and Montgomery-Smith~\cite{HM}
in the Banach space literature; and
(iii) the result holds in the most primitive setting needed to state it,
thereby being applicable also to e.g.~Lie groups. 

We now discuss several ways in which to estimate the size of sums of
independent $\G$-valued random variables, for metric semigroups $\G$. We
present two results in this section, corresponding to two of the
estimation techniques discussed in the introduction.
(For a third result, see Theorem~\ref{Tbounds}.)

The \textbf{first} approach, informally speaking, uses the
Hoffmann-J{\o}rgensen inequality to generalize an upper bound for
$\be_\mu[\|S_n\|^p]$ in terms of the quantiles of $\|S_n\|$ as well as
$\be_\mu[M_n^p]$ -- but now in the ``minimal'' framework of metric
semigroups. More precisely, we show that controlling the behavior of
$X_n$ is equivalent to controlling $S_n$ or $U_n$, for all metric
semigroups.

\begin{utheorem}\label{Thj2}
Suppose $A \subset \N$ is either $\N$ or $\{ 1, \dots, N \}$ for some $N
\in \N$. Suppose $(\G, d_\G)$ is a separable metric semigroup, $z_0, z_1
\in \G$, and $X_n \in L^0(\Omega,\G)$ are independent for all $n \in A$.
If $\sup_{n \in A} d_\G(z_1, z_0 S_n) < \infty$ almost surely, then for
all $p \in (0,\infty)$,
\[
\be_\mu \left[ \sup_{n \in A} d_\G(z_0, z_0 X_n)^p \right] < \infty \quad
\Longleftrightarrow \quad \be_\mu \left[ \sup_{n \in A} d_\G(z_1, z_0
S_n)^p \right] < \infty.
\]
\end{utheorem}

This result extends \cite[Theorem 3.1]{HJ} by Hoffmann-J{\o}rgensen to
the ``minimal'' framework of metric semigroups. The proofs of
Theorem~\ref{Thj2} and the next result use the notion of the quantile
functions, or decreasing rearrangements, of $\G$-valued random variables:

\begin{definition}\label{Ddecrear}
Suppose $(\G, d_\G)$ is a metric semigroup, and $X : (\Omega,
\mathscr{A}, \mu) \to (\G,\mathscr{B}_\G)$. We define the
\textit{decreasing} (or \textit{non-increasing}) \textit{rearrangement}
of $X$ to be the right-continuous inverse $X^*$ of the function $t
\mapsto \bp{d_\G(z_0, z_0 X) > t}$, for any $z_0 \in \G$. In other words,
$X^*$ is the real-valued random variable defined on $[0,1]$ with the
Lebesgue measure, as follows:
\[
X^*(t) := \sup \{ y \in [0,\infty) : \bp{d_\G(z_0, z_0 X) > y} > t \}.
\]
\end{definition}

Note that $X^*$ has exactly the same law as $d_\G(z_0, z_0 X)$. Moreover,
if $(\G, \| \cdot \|)$ is a normed linear space, then $d_\G(z_0, z_0 X)$
can be replaced by $\|X\|$, and often papers in the literature refer to
$X^*$ as the decreasing rearrangement of $\|X\|$ instead of $X$ itself.
The convention that we adopt above is slightly weaker.

The \textbf{second} approach provides another estimate on the size of
$S_n$ through its moments, by comparing $\| S_n \|_q$ to $\| S_n \|_p$ --
or more precisely, $\be_\mu[U_n^q]^{1/q}$ to $\be_\mu[U_n^p]^{1/p}$ --
for $0 < p \leqslant q$. Moreover, the constants of comparison are
universal, valid for all abelian semigroups and all finite sequences of
independent random variables, and depend only on a threshold:

\begin{utheorem}\label{Tupq}
Given $p_0 > 0$, there exist universal constants $c = c(p_0), c' =
c'(p_0) > 0$ depending only on $p_0$, such that for all choices of
(a) separable abelian metric semigroups $(\G, d_\G)$,
(b) finite sequences of independent $\G$-valued random variables $X_1,
\dots, X_n$,
(c) $q \geqslant p \geqslant p_0$, and
(d) $\epsilon \in (-q,\log(16)]$,
we have
\begin{align*}
\be_\mu[U_n^q]^{1/q} \leqslant &\ c \frac{q}{\max(p, \log(\epsilon+q))} (
\be_\mu[U_n^p]^{1/p} + M_n^*(e^{-q}/8)) + c \be_\mu[M_n^q]^{1/q}\\
\leqslant &\ c' \frac{q}{\max(p, \log(\epsilon+q))} (\be_\mu[U_n^p]^{1/p}
+ \be_\mu[M_n^q]^{1/q}) \quad \text{ if } \epsilon \geqslant \min(1,
e-p_0).
\end{align*}

\noindent Moreover, we may choose
\[
c'(p_0) = c(p_0) \cdot \left( 8^{1/p_0}e + \max(1,
\frac{\log(\epsilon+p_0)}{p_0}) \right).
\]
\end{utheorem}

Theorem~\ref{Tupq} extends a host of results in the Banach space
literature, including by Johnson--Schechtman--Zinn~\cite{JSZ},
Hitczenko~\cite{Hi}, and Hitczenko and Montgomery-Smith~\cite{HM}. (See
also \cite[Theorem 6.20]{LT} and \cite[Proposition 1.4.2]{KW}.)
Theorem~\ref{Tupq} also yields the correct order of the constants as $q
\to \infty$, as discussed by Johnson \textit{et al} in
\textit{loc.~cit.}~where they extend previous work on Khinchin's
inequality by Rosenthal \cite{Ro}. Moreover, all of these results are
shown for Banach spaces. Theorem~\ref{Tupq} holds additionally for all
compact Lie groups, finite abelian groups and lattices, and spaces of
labelled and unlabelled graphs.

\section{L\'evy's equivalence in metric semigroups}\label{Slevy}

In this section we prove:

\begin{theorem}[L\'evy's Equivalence]\label{Tlevy}
Suppose $(\G, d_\G)$ is a complete separable metric semigroup, $X_n :
(\Omega, \mathscr{A}, \mu) \to (\G, \mathscr{B}_\G)$ are independent, $X
\in L^0(\Omega, \G)$, and $S_n$ is defined as in \eqref{Esemigroup}. Then
\[
S_n \longrightarrow X\ a.s.~\mathbb{P}_\mu \quad \Longleftrightarrow
\quad S_n \conv{P} X.
\]

\noindent Moreover, if the sequence $S_n$ does not converge as above,
then it diverges almost surely.
\end{theorem}

Special cases of this result have been shown in the literature. For
instance, \cite[\S 9.7]{Du} considers $\G = \mathbb{R}^n$. The more
general case of a separable Banach space $\bb$ was shown by
It$\hat{\mbox{o}}$--Nisio \cite[Theorem 3.1]{IN}, as well as by
Hoffmann-J{\o}rgensen and Pisier \cite[Lemma 1.2]{HJP}.
The most general version in the literature to date is by Tortrat, who
proved the result for a complete separable metric group in \cite{To2}.
Thus Theorem~\ref{Tlevy} is the closest to assuming only the minimal
structure necessary to state the result (as well as to prove it).

In order to prove Theorem~\ref{Tlevy}, we first study basic properties of
metric semigroups. Note that for a metric group, the following is
standard; see \cite{Kl}, for instance.

\begin{lemma}\label{Linverse}
If $(\G, d_\G)$ is a metric (semi)group, then the translation-invariance
of $d_\G$ implies the ``triangle inequality'':
\begin{equation}\label{Etriangle}
d_\G(y_1 y_2, z_1 z_2) \leqslant d_\G(y_1, z_1) + d_\G(y_2, z_2)\ \forall
y_i, z_i \in \G,
\end{equation}

\noindent and in turn, this implies that each (semi)group operation is
continuous.

If instead $\G$ is a group equipped with a metric $d_\G$, then except for
the last two statements, any two of the following assertions imply the
other two:
\begin{enumerate}
\item $d_\G$ is left-translation invariant: $d_\G(ca, cb) = d_\G(a,b)$
for all $a,b,c \in \G$. In other words, left-multiplication by any $c \in
\G$ is an isometry.

\item $d_\G$ is right-translation invariant.

\item The inverse map $: \G \to \G$ is an isometry. Equivalently, the
triangle inequality~\eqref{Etriangle} holds.

\item $d_\G$ is invariant under all inner/conjugation automorphisms.
\end{enumerate}
\end{lemma}

In order to show Theorem~\ref{Tlevy} for metric semigroups, we recall the
following preliminary result from \cite{KR4}, and will use it below
without further reference.

\begin{prop}[\cite{KR4}]\label{Cstrict}
Suppose $(\G, d_\G)$ is a metric semigroup, and $a,b \in \G$. Then
\begin{equation}\label{Estrict}
d_\G(a,ba) = d_\G(b,b^2) = d_\G(a,ab)
\end{equation}

\noindent is independent of $a \in \G$. Moreover, a set $\G$ is a metric
semigroup only if $\G$ is a metric monoid, or the set of non-identity
elements in a metric monoid $\G'$. This is if and only if the number of
idempotents in $\G$ is one or zero, respectively.
Furthermore, the metric monoid $\G'$ is (up to a monoid isomorphism) the
unique smallest element in the class of metric monoids containing $\G$ as
a sub-semigroup.
\end{prop}

\begin{remark}\label{Rstrict}
In the sequel, we denote -- when required -- the unique metric monoid
containing a given metric semigroup $\G$ by $\G' := \G \cup \{ 1' \}$.
Note that the idempotent $1'$ may already be in $\G$, in which case $\G =
\G'$. One consequence of Proposition~\ref{Cstrict} is that instead of
working with metric semigroups, one can use the associated monoid $\G'$
instead. (In other words, the (non)existence of the identity is not an
issue in many such cases.)
This helps simplify other calculations. For instance, what would be a
lengthy, inductive (yet straightforward) computation now becomes much
simpler: for nonnegative integers $k,l$, and $z_0, z_1, \dots, z_{k+l}
\in \G$, the triangle inequality~\eqref{Etriangle} implies:
\[
d_\G(z_0 \cdots z_k, z_0 \cdots z_{k+l}) = d_{\G'}(1', \prod_{i=1}^l
z_{k+i}) \leqslant \sum_{i=1}^l d_{\G'}(1', z_{k+i}) = \sum_{i=1}^l
d_\G(z_0, z_0 z_{k+i}).
\]
\end{remark}

\subsection{The Mogul'skii inequalities and proof of L\'evy's
equivalence}

Like L\'evy's Equivalence (Theorem~\ref{Tlevy}) and the
Hoffmann-J{\o}rgensen inequality (Theorem~\ref{Thj}), many other maximal
and minimal inequalities can be formulated using only the notions of a
distance function and of a semigroup operation. We now extend to metric
semigroups two inequalities by Mogul'skii, which were used in \cite{Mog}
to prove a law of the iterated logarithm in normed linear spaces. The
following result will be useful in proving Theorem~\ref{Tlevy}.

\begin{prop}[Mogul'skii--Ottaviani--Skorohod inequalities]\label{Pmog}
Suppose $(\G, d_\G)$ is a separable metric semigroup, $z_0, z_1 \in \G$,
$a,b \in [0,\infty)$, and $X_1, \dots, X_n \in L^0(\Omega,\G)$ are
independent. Then for all integers $1 \leqslant m \leqslant n$,
\begin{align*}
& \bp{\min_{m \leqslant k \leqslant n} d_\G(z_1, z_0 S_k) \leqslant a}
\cdot \min_{m \leqslant k \leqslant n} \bp{d_\G(S_k, S_n) \leqslant b}
\leqslant \bp{d_\G(z_1, z_0 S_n) \leqslant a + b},\\
& \bp{\max_{m \leqslant k \leqslant n} d_\G(z_1, z_0 S_k) \geqslant a}
\cdot \min_{m \leqslant k \leqslant n} \bp{d_\G(S_k, S_n) \leqslant b}
\leqslant \bp{d_\G(z_1, z_0 S_n) \geqslant a - b}.
\end{align*}
\end{prop}

These inequalities strengthen \cite[Lemma 1]{Mog} from normed linear
spaces to arbitrary metric semigroups. Also note that the second
inequality generalizes the \textit{Ottaviani--Skorohod inequality} to all
metric semigroups. Indeed, sources such as \cite[\S 9.7.2]{Du} prove this
result in the special case $\G = (\mathbb{R}^n, +), z_0 = z_1 = 0, m=1, a
= \alpha + \beta, b = \beta$, with $\alpha, \beta > 0$.

We omit the proof of Proposition \ref{Pmog} for brevity as it involves
standard arguments. Using this result, one can now prove
Theorem~\ref{Tlevy}. The idea is to use the approach in \cite{Du};
however, it needs to be suitably modified in order to work in the current
level of generality.

\begin{proof}[Proof of Theorem~\ref{Tlevy}]
The forward implication is easily verified in the more general setting of
a separable metric space; see e.g.~\cite[Section 9.2]{Du}. Conversely, we
claim that $S_i$ is Cauchy almost everywhere, if it converges in
probability to $X$. Given $\epsilon,\eta > 0$, the assumption and
definitions imply that there exists $n_0 \in \N$ such that
\[
\bp{d_\G(S_m,X) \geqslant \epsilon/8} < \frac{\eta}{2(1 + \eta)}, \quad
\forall m \geqslant n_0.
\]

\noindent This implies that
$\displaystyle \bp{d_\G(S_m,S_n) \geqslant \epsilon / 4} < \frac{\eta}{1
+ \eta}$ for all $n \geqslant m \geqslant n_0$.
Now define $S'_i := \prod_{j=1}^i X_{n_0+j}$. Fix $n > n_0$ and apply
Proposition~\ref{Pmog} to $\{ X_{n_0+i} : 1 \leqslant i \leqslant n - n_0
\}$ with $m=1, a = \epsilon/2, b = \epsilon/4$, and $z_0 = z_1$:
\begin{align*}
& \bp{\max_{n_0 + 1 \leqslant m \leqslant n} d_\G(S_{n_0}, S_m) \geqslant
\epsilon/2} = \bp{\max_{1 \leqslant i \leqslant n - n_0} d_{\G'}(z_0, z_0
S'_i) \geqslant \epsilon/2}\\
\leqslant &\ \frac{\bp{d_{\G'}(z_0, z_0 S'_{n-n_0}) \geqslant
\epsilon/4}}{1 - \max_{1 \leqslant i \leqslant n-n_0}
\bp{d_{\G'}(S'_i,S'_{n-n_0}) \geqslant \epsilon/4}} < \frac{\eta/ (1 +
\eta)}{1 - \eta/(1+\eta)} = \eta.
\end{align*}

Now define $Q_{n_0} := \sup_{n > n_0} d_\G(S_{n_0}, S_n)$ and
$\delta_{n_0} := \sup_{n > m > n_0} d_\G(S_m,S_n)$. Then $\delta_{n_0}
\leqslant 2 Q_{n_0}$; moreover, taking the limit of the above inequality
as $n \to \infty$ yields:
\[
\bp{Q_{n_0} \geqslant \epsilon / 2} \leqslant \eta \quad \implies \quad
\bp{\delta_{n_0} \geqslant \epsilon} \leqslant \eta.
\]

\noindent But then $\bp{\sup_{n > m} d_\G(S_m,S_n) \geqslant \epsilon}
\leqslant \eta$ for all $m > n_0$. Thus, $S_n$ is Cauchy almost
everywhere. Since $\G$ is complete, the result now follows from
\cite[Lemma 9.2.4]{Du}; that the almost sure limit is $X$ is because $S_n
\conv{P} X$. Finally, since the $X_n$ are independent, the
convergence of the sequence $S_n$ is a tail event. In particular, it has
probability zero or one by the Kolmogorov 0-1 law, concluding the proof.
\end{proof}

We remark for completeness that the other L\'evy equivalence has been
addressed in \cite{Csi,Ga,To2} for various classes of topological groups.
See also \cite{MS} for a variant in discrete completely simple
semigroups, \cite{Du,IN} for Banach space versions, and \cite{KR4} for a
version over any normed abelian metric (semi)group.

\section{Measuring the magnitude of sums of independent random
variables}\label{Stails}

We now prove Theorems~\ref{Thj2} and~\ref{Tupq} using the
Hoffmann-J{\o}rgensen inequality in Theorem~\ref{Thj}.
Recall that the Banach space version of this inequality is extremely
important in the literature and is widely used in bounding sums of
independent Banach space-valued random variables. Having proved
Theorem~\ref{Thj}, an immediate application of our main result is in
obtaining the first such bounds for general metric semigroups $\G$. We
also provide uniformly good $L^p$-bounds and tail probability bounds on
sums $S_n$ of independent $\G$-valued random variables.

\subsection{An upper bound by Hoffmann-J{\o}rgensen}

In this subsection we prove Theorem~\ref{Thj2}. The proof uses basic
properties of decreasing rearrangements (see Definition~\ref{Ddecrear}),
which we record here and use below, possibly without reference.

\begin{prop}\label{Pdecrear}
Suppose $X, Y : (\Omega, \mathscr{A}, \mu) \to [0,\infty)$ are random
variables, and
\[ x,\alpha,\beta,\gamma > 0, \quad t \in [0,1]. \]
\begin{enumerate}
\item 
$X^*(t) \leqslant x$ if and only if $\bp{X > x} \leqslant t$.

\item $X^*(t)$ is decreasing in $t \in [0,1]$ and increasing in $X
\geqslant 0$.

\item $(X/x)^*(t) = X^*(t)/x$.

\item Suppose $\bp{X > x} \leqslant \beta \bp{Y > \gamma x}$ for all
$x>0$. Then for all $p \in (0,\infty)$ and $t \in (0,1)$,
\[
\be_\mu[Y^p] \geqslant \beta^{-1} \gamma^p \be_\mu[X^p], \qquad
\be_\mu[X^p] \geqslant t X^*(t)^p.
\]

\item Fix finitely many tuples of positive constants $(\alpha_i, \beta_i,
\gamma_i, \delta_i)_{i=1}^N$, and real-valued nondecreasing functions
$f_i$ such that for all $x>0$ there exists at least one $i$ such that
\begin{equation}\label{Edecrear0}
f_i(\bp{X > \alpha_i x}) \leqslant \beta_i \bp{Y > \gamma_i
x}^{\delta_i}.
\end{equation}

\noindent Then
\begin{equation}\label{Edecrear}
X^*(t) \leqslant \max_{1 \leqslant i \leqslant N}
\frac{\alpha_i}{\gamma_i} Y^*((f_i(t)/\beta_i)^{1/\delta_i}).
\end{equation}

\noindent If on the other hand~\eqref{Edecrear0} holds for all $i$, then
$\displaystyle X^*(t) \leqslant \min_{1 \leqslant i \leqslant N}
\frac{\alpha_i}{\gamma_i} Y^*((f_i(t)/\beta_i)^{1/\delta_i})$.
\end{enumerate}
\end{prop}

\begin{proof}
These properties are shown using the definitions via
straightforward arguments, and so we omit the proofs, except for the
final part. By assumption there exists at least one $i$ such that if
$\bp{X > \alpha_i x} > t$ for some $t$, then $\beta_i \bp{Y > \gamma_i
x}^{\delta_i} > f_i(t)$ since $f_i$ is nondecreasing. For this choice of
$i$, we obtain:
\begin{align*}
\alpha_i^{-1} \{ y : \ \bp{X > y} > t \} \subset &\ \gamma_i^{-1} \{ y :
\ \beta_i \bp{Y > y}^{\delta_i} > f_i(t) \}\\
= &\ \gamma_i^{-1} \{ y : \ \bp{Y > y} > (f_i(t)/\beta_i)^{1/\delta_i} \}
\end{align*}

\noindent (where we only consider $y \geqslant 0$). Therefore for all $t
\in [0,1]$,
\[
\{ y \geqslant 0 : \ \bp{X > y} > t \} \quad \subset \quad
\bigcup_{i=1}^N\ \frac{\alpha_i}{\gamma_i} \{ y \geqslant 0 : \ \bp{Y >
y} > (f_i(t)/\beta_i)^{1/\delta_i} \}.
\]

\noindent Taking the supremum of both sides yields
Equation~\eqref{Edecrear}. If on the other hand
Equation~\eqref{Edecrear0} holds for all $i$, then the preceding
inclusion holds with the union replaced by intersection. Now taking the
supremum of both sides yields Equation~\eqref{Edecrear} with maximum
replaced by minimum (since each set in the intersection is an interval
containing $0$).
\end{proof}

Using Proposition~\ref{Pdecrear}, we now show one of the main results in
this paper.

\begin{proof}[Proof of Theorem~\ref{Thj2}]
Note for all $n$ that
\[
d_\G(z_0, z_0 X_n) \leqslant d_\G(z_1, z_0 S_{n-1}) + d_\G(z_1, z_0 S_n),
\]
from which we obtain
\[
d_\G(z_0, z_0 X_n)^p \leqslant 2^{p+1} \sup_{n \in A} d_\G(z_1, z_0 S_n)^p.
\]

\noindent Taking first the supremum over $n \in A$ and then the
expectation proves the backward implication.
Conversely, first claim that controlling sums of $\G$-valued $L^p$ random
variables in probability (i.e., in $L^0$) allows us to control these sums
in $L^p$ as well, for $p>0$. Namely, we make the following claim:\medskip

\textit{Suppose $(\G, d_\G)$ is a separable metric semigroup, $p \in
(0,\infty)$, and $X_1, \dots, X_n$ $\in L^p(\Omega,\G)$ are independent.
Now fix $z_0, z_1 \in \G$ and let $S_k, U_n, M_n$ be as in
Definition~\ref{Dsemi} and Equation~\eqref{Esemigroup}. Then,}
\[
\be_\mu[U_n^p] \leqslant 2^{1 + 2p} (\be_\mu[M_n^p] +
U_n^*(2^{-1-2p})^p).
\]

Note that the claim is akin to the upper bound by Hoffmann-J{\o}rgensen
that bounds $\be_\mu[\| S_n \|^p]$ in terms of $\be_\mu[M_n^p]$ and the
quantiles of $\| S_n \|$ for Banach space-valued random variables  (see
\cite[proof of Theorem 3.1]{HJ} and \cite[Lemma 3.1]{GZ}).
We omit its proof for brevity, as a similar statement is asserted in
\cite[Proposition 6.8]{LT}.
Given the claim, define:
\begin{align}\label{Edefma}
\begin{aligned}
t_n := & U_n^*(2^{-1-2p}) \quad (n \in A),\qquad
U_A := \sup_{n \in A} d_\G(z_1, z_0 S_n), \\
M_A := &\ \sup_{n \in A} d_\G(z_0, z_0 X_n), \qquad \quad
t_A := U_A^*(2^{-1-2p}),
\end{aligned}
\end{align}

\noindent as above, where we also use the assumption that $U_A < \infty$
almost surely. Now for all $n \in A$, compute using the above claim and
elementary properties of decreasing rearrangements:
\[
\be_\mu[U_n^p] \leqslant 2^{1+2p} \be_\mu[M_n^p] + 2 (4 t_n)^p \leqslant
2^{1+2p} \be_\mu[M_A^p] + 2 (4 t_A)^p.
\]

This concludes the proof if $A$ is finite; for $A = \N$, use the Monotone
Convergence Theorem for the increasing sequence $0 \leqslant U_n^p \to
U_A^p$.
\end{proof}

\subsection{Two-sided bounds and $L^p$ norms}

We now formulate and prove additional results that control tail behavior
for metric semigroups and monoids -- specifically, $M_A, U_n, U_n^*$.
This includes proving our other main result, Theorem~\ref{Tupq}. We begin
by setting notation.

\begin{definition}\label{Dtrunc}
Suppose $\G$ is a metric semigroup.
\begin{enumerate}
\item Given $X_n \in L^0(\Omega,\G)$ as above, for all $n$ in a finite or
countable set $A$, define the random variable $\ell_X = \ell_{(X_n)} :
\mathbb{R} \to [0,\infty]$ via:
\[
\ell_X(t) := \begin{cases}
\inf \{ y > 0\ : \ \sum_{n \in A} \bp{d_\G(z_0, z_0 X_n) >
y} \leqslant t \}, \qquad & \text{if } t \in [0,1],\\
0, & \text{otherwise.}
\end{cases}
\]
As indicated in \cite[\S 2]{HM}, one then has: $\mathbb{P}(\ell_X > x) =
\sum_{n \in A} \bp{d_\G(z_0, z_0 X_n) > x}$, where $\mathbb{P}$ is the
Lebesgue measure on $[0,1]$.

\item Two families of variables $P(t)$ and $Q(t)$ are said to be
\textit{comparable}, denoted by $P(t) \approx Q(t)$, if there exist
constants $c_1, c_2 > 0$ such that $c_1^{-1} P(t) \leqslant Q(t)
\leqslant c_2 P(t)$ uniformly over all $t$. The $c_i$ are called the
``constants of approximation''.\medskip

\noindent \textit{For the remaining definitions, assume $(\G, 1_\G,
d_\G)$ is a separable metric monoid.}\medskip

\item Given $t \geqslant 0$ and a random variable $X
\in L^0(\Omega, \G)$, define its \textit{truncation} to be:
\[
X(t) := \begin{cases} 1_\G, \qquad & \mbox{ if } d_\G(1_\G, X) > t,\\
X, & \mbox{ otherwise.}
\end{cases}
\]

\item Given variables $X_1, \dots, X_n : \Omega \to \G$, and $r \in
(0,1)$, define:
\[
U'_n(r) := \max_{1 \leqslant k \leqslant n} d_\G(1_\G, \prod_{i=1}^k
X_i(\ell_X(r))).
\]
\end{enumerate}
\end{definition}

The following estimate on tail behavior compares $U_n$ with its
decreasing rearrangement.

\begin{theorem}\label{Tbounds}
Given $p_0 > 0$, there exist universal constants of approximation
(depending only on $p_0$), such that for all $p \geqslant p_0$, separable
abelian metric monoids $(\G, 1_\G, d_\G)$, and finite sequences $X_1,
\dots, X_n$ of independent $\G$-valued random variables (for any $n \in
\N$),
\[
\be_\mu[U_n^p]^{1/p} \approx U_n^*(e^{-p}/4) + \be[\ell_X^p]^{1/p}
\approx (U'_n(e^{-p}/8))^* (e^{-p}/4) + \be[\ell_X^p]^{1/p},
\]

\noindent where $U_n$ and $U'_n$ were defined in Equation
\eqref{Esemigroup} and Definition \ref{Dtrunc} respectively.
\end{theorem}

\noindent For real-valued $X$, the expression $\be[|X|^p]^{1/p}$ is also
denoted by $\| X \|_p$ in the literature.

To show Theorem~\ref{Tbounds}, we require some preliminary results which
provide additional estimates to govern tail behavior, and which we now
collect before proving the theorem. As these preliminaries are often
extensions to metric semigroups of results in the Banach space
literature, we will sketch or omit their proofs now.

The first result obtains two-sided bounds to control the behavior of the
``maximum magnitude'' $M_A$ (cf.~Equation~\eqref{Edefma}).

\begin{prop}\label{Pbounds}
Suppose $\{ X_n : n \in A \}$ is a (finite or countably infinite)
sequence of independent random variables with values in a separable
metric semigroup $(\G, d_\G)$.
\begin{enumerate}
\item For all $t \in (0,1)$, $\ell_X(2t) \leqslant \ell_X(t/(1-t))
\leqslant M_A^*(t) \leqslant \ell_X(t)$.

\item Suppose $X_n \in L^p(\Omega,\G)$ for some $p>0$ (and for all $n \in
A$). For all $t > 0$, define:
\[
\Psi_X(t) := p \sum_{n \in A} \int_{\ell_X(t)}^\infty u^{p-1}
\bp{d_\G(z_0, z_0 X_n) > u}\ du.
\]

\noindent Then, $\displaystyle \frac{t \ell_X(t)^p + \Psi_X(t)}{1+t}
\leqslant \be_\mu[M_A^p] \leqslant \ell_X(t)^p + \Psi_X(t)$.
\end{enumerate}
\end{prop}

\begin{proof}
The first part follows \cite[Proposition 1]{HM} (using a special case of
Equation~\eqref{Edecrear}). For the second, follow the arguments for
showing \cite[Lemma 3.2]{GZ}; see also \cite[Lemma 6.9]{LT}.
\end{proof}

We next discuss a consequence of Hoffmann-J{\o}rgensen's inequality for
metric semigroups, Theorem~\ref{Thj}, which can be used to bound the
$L^p$-norms of the variables $U_n$ -- or more precisely, to relate these
$L^p$-norms to the tail distributions of $U_n$ via $U_n^*$.

\begin{lemma}\label{Lbounds}
(Notation as in Definition~\ref{Dsemi} and Equation~\eqref{Esemigroup}.)
There exists a universal positive constant $c_1$ such that for any $0
\leqslant t \leqslant s \leqslant 1/2$, any separable metric semigroup
$(\G, d_\G)$ with elements $z_0, z_1$, and any sequence of independent
$\G$-valued random variables $X_1, \dots, X_n$,
\[
U_n^*(t) \leqslant c_1 \frac{\log(1/t)}{\max \{ \log(1/s), \log \log(4/t)
\}} ( U_n^*(s) + M_n^*(t/2)).
\]
\end{lemma}

\begin{proof}
We begin by writing down a consequence of Theorem~\ref{Thj}:
\begin{equation}\label{Eunr}
\bp{U_n > (3K-1)t} \leqslant \frac{1}{K!} \left( \frac{\bp{U_n >
t}}{\bp{U_n \leqslant t}} \right)^K + \bp{M_n > t}, \quad \forall t > 0,\
\forall K,n \in \N.
\end{equation}

\noindent If $\bp{U_n > t} \leqslant 1/2$, then this quantity is further
dominated by
\[
2 \max \left\{ \bp{M_n > t}, \frac{1}{K!} (2 \bp{U_n > t})^K \right\}.
\]

\noindent Now carry out the steps mentioned in the proof of
\cite[Corollary 1]{HM}.
\end{proof}

The final preliminary result is proved by adapting the proofs of
\cite[Lemma 3 and Corollary 2]{HM} to metric monoids.

\begin{prop}\label{Pbounds2}
Suppose $(\G, 1_\G, d_\G)$ is a separable metric monoid and $X_1, \dots,
X_n: \Omega \to \G$ is a finite sequence of independent $\G$-valued
random variables. For $r \in (0,1)$, define:
\[
U''_n(r) := \max_{1 \leqslant k \leqslant n} d_\G(1_\G, \prod_{i=1}^k
X_i'(\ell_X(r))),
\]

\noindent where $X'_i(t)$ equals $1_\G$ if $d_\G(1_\G, X_i) \leqslant t$,
and $X_i$ otherwise.
\begin{enumerate}
\item Then $U''_n(r)$ may be expressed as the sum of ``disjoint'' random
variables $V_k$ for $k \in \N$. In other words, $\Omega$ can be
partitioned into measurable subsets $E_k$ such that $V_k = U''_n(r)$ on
$E_k$ and $1_\G$ otherwise. Moreover, the $V_k$ may be chosen such that
$V_k^*(t) \leqslant k \cdot \ell(t (k-1)! / r^{k-1})$.

\item Given the assumptions, for all $p \in (0,\infty)$,
\[
\be_\mu[U''_n(r)^p]^{1/p} \leqslant 2 e^{2^p r/p}
\be[\ell_X^p]^{1/p}.
\]
\end{enumerate}
\end{prop}

With the above results in hand, we can now show the above theorem.

\begin{proof}[Proof of Theorem~\ref{Tbounds}]
Compute using the triangle inequality~\eqref{Etriangle} and
Remark~\ref{Rstrict}:
\[
d_\G(1_\G, X_k) \leqslant d_\G(1_\G, S_{k-1}) + d_\G(1_\G, S_k) \leqslant
2 U_n.
\]

\noindent Hence $M_n \leqslant 2 U_n$. Now compute for $p \geqslant p_0$,
using Propositions~\ref{Pdecrear} and \ref{Pbounds}:
\begin{align*}
\be_\mu[U_n^p]^{1/p} \geqslant &\ \frac{1}{2} \be_\mu[M_n^p]^{1/p}
\geqslant 2^{-1-p_0^{-1}} \be[\ell_X^p]^{1/p},\\
\be_\mu[U_n^p]^{1/p} \geqslant &\ (e^{-p}/8)^{1/p} U_n^*(e^{-p}/8)
\geqslant 8^{-p_0^{-1}} e^{-1} U_n^*(e^{-p}/4).
\end{align*}

\noindent Hence there exists a constant $0 < c_1 = c_1(p_0)$ such that:
\[
\be_\mu[U_n^p]^{1/p} \geqslant c_1^{-1} (U_n^*(e^{-p}/4) +
\be[\ell_X^p]^{1/p}).
\]

This yields one inequality; another one is obtained using
Proposition~\ref{Pbounds} as follows:
\[
\bp{U_n \neq U'_n(e^{-p}/8)} \leqslant \mathbb{P}(M_n > \ell_X(e^{-p}/8))
\leqslant \bp{M_n > M_n^*(e^{-p}/8)} \leqslant e^{-p}/8.
\]

\noindent Now if
$\bp{U'_n(e^{-p}/8) > y} > \eta$ for some $\eta \in
[\frac{e^{-p}}{8},1]$, then by the reverse triangle inequality,
\begin{align*}
\bp{U_n > y} \geqslant &\ \bp{U_n > y,\ U_n = U'_n(e^{-p}/8)}\\
\geqslant &\ \bp{U'_n(e^{-p}/8) > y} - \bp{U_n \neq U'_n(e^{-p}/8)}
> \eta - \frac{e^{-p}}{8}.
\end{align*}

\noindent Hence by definition and the above calculations,
\begin{equation}\label{Ehm}
U'_n(e^{-p}/8)^*(\eta) \leqslant U_n^*(\eta - e^{-p}/8).
\end{equation}

\noindent Applying this with $\eta = e^{-p}/4$,
\[
U'_n(e^{-p}/8)^*(e^{-p}/4) \leqslant U_n^*(e^{-p}/8) \leqslant e 8^{1/p}
\be_\mu[U_n^p]^{1/p} \leqslant e 8^{1/p_0} \be_\mu[U_n^p]^{1/p}.
\]

\noindent Hence as above, there exists a constant $0 < c_2 = c_2(p_0)$
such that:
\[
\be_\mu[U_n^p]^{1/p} \geqslant c_2^{-1} ( U'_n(e^{-p}/8)^*(e^{-p}/4) +
\be[\ell_X^p]^{1/p}).
\]

\noindent This proves the second of the four claimed inequalities. The
remaining arguments can now be shown by suitably adapting the proof of
\cite[Theorem 3]{HM}.
\end{proof}

Finally, we use Theorem~\ref{Tbounds} to prove our remaining main result.

\begin{proof}[Proof of Theorem~\ref{Tupq}]
Using Proposition~\ref{Cstrict}, let $\G'$ denote the smallest metric
monoid containing $\G$. Thus the $X_k$ are a sequence of independent
$\G'$-valued random variables, and we may assume henceforth that $\G =
\G'$. Compute using Proposition~\ref{Pbounds}, and the fact that $X^*$
and $X$ have the same law for the real-valued random variable $X = M_n$:
\begin{align*}
\be[\ell_X^q] = &\ \int_0^{1/2} \ell_X(2t)^q \cdot 2 dt \leqslant 2
\int_0^{1/2} M^*_n(t)^q\ dt \leqslant 2 \int_0^1 M^*_n(t)^q\ dt = 2
\mathbb{E}[(M^*_n)^q]\\
= &\ 2 \be_\mu[M_n^q].
\end{align*}

Using this computation, as well as Lemma~\ref{Lbounds} and
Theorem~\ref{Tbounds} for $\G'$, we compute:
\begin{align*}
&\ \be_\mu[U_n^q]^{1/q}\\
\leqslant &\ c'_1 (\be[\ell_X^q]^{1/q} + U_n^*(e^{-q}/4))\\
\leqslant &\ c'_1 \cdot 2^{1/q} \be_\mu[M_n^q]^{1/q} + c'_1 c_1
\frac{\log(4e^q)}{\max(\log(4e^p), \log \log (16e^q))} (U_n^*(e^{-p}/4) +
M_n^*(e^{-q}/8))\\
\leqslant &\ c'_1 \cdot 2^{1/q} \be_\mu[M_n^q]^{1/q} + c'_1 c_1
\frac{\log(4e^q)}{\max(\log(4e^p), \log(\epsilon + q))} (c_2
\be_\mu[U_n^p]^{1/p} + M^*_n(e^{-q}/8))
\end{align*}

\noindent since $\epsilon \in (-q, \log(16)]$. There are now two cases:
first if $e^p \geqslant \epsilon+q$, then
\[
\frac{\log(4e^q)}{\max(\log(4e^p), \log(\epsilon + q))} \leqslant \frac{q
+ \log(4)}{p + \log(4)} \leqslant \frac{q}{p} = \frac{q}{\max(p,
\log(\epsilon + q))}.
\]

\noindent On the other hand, if $e^p < \epsilon + q$ then set $C := 1 +
\frac{\log(4)}{p_0}$ and note that $Cq \geqslant q + \log(4)$. Therefore,
\[
\frac{\log(4e^q)}{\max(\log(4e^p), \log(\epsilon + q))} \leqslant \frac{q
+ \log(4)}{\log(\epsilon + q)} \leqslant \frac{C q}{\log(\epsilon + q)} =
C \frac{q}{\max(p, \log(\epsilon + q))}.
\]

Using the above analysis now yields:
\begin{align*}
&\ \be_\mu[U_n^q]^{1/q}\\
\leqslant &\ c'_1 \cdot 2^{1/q} \be_\mu[M_n^q]^{1/q} + c'_1 c_1 \left( 1
+ \frac{\log(4)}{p_0} \right) \frac{q}{\max(p, \log(\epsilon+q))} (c_2
\be_\mu[U_n^p]^{1/p} + M^*_n(e^{-q}/8)).
\end{align*}

\noindent Setting $c := c'_1 \max(2^{1/p_0}, c_1(1 + \log(4)/p_0), c_1
c_2 (1 + \log(4)/p_0))$, we obtain the first inequality claimed in the
statement of the theorem.

To show the second inequality, we first verify that if $\epsilon
\geqslant \min(1, e - p_0)$, then the function $f(x) :=
x/\log(\epsilon+x)$ is strictly increasing on $(p_0,\infty)$. 
%
%
Now compute:
\begin{align*}
\frac{q}{\max(p, \log(\epsilon+q))} = &\ \min \left( \frac{q}{p},
\frac{q}{\log(\epsilon+q)} \right) \geqslant \min \left( 1,
\frac{q}{\log(\epsilon+q)} \right)\\
\geqslant &\ \min \left( 1, \frac{p_0}{\log(\epsilon+p_0)} \right).
\end{align*}

Next, use Proposition~\ref{Pdecrear} to show: $M^*_n(e^{-q}/8)
\leqslant \be_\mu[M_n^q]^{1/q} (8e^q)^{1/q} \leqslant 8^{1/p_0} e
\be_\mu[M_n^q]^{1/q}$. Using the previous two facts, we now complete the
proof of the second inequality by beginning with the first inequality:
\begin{align*}
&\ \be_\mu[U_n^q]^{1/q}\\
\leqslant &\ c \frac{q}{\max(p, \log(\epsilon+q))} ( \be_\mu[U_n^p]^{1/p}
+ M_n^*(e^{-q}/8)) + c \be_\mu[M_n^q]^{1/q}\\
\leqslant &\ c \frac{q}{\max(p, \log(\epsilon+q))} (\be_\mu[U_n^p]^{1/p}
+ 8^{1/p_0} e \be_\mu[M_n^q]^{1/q}) + c  \cdot 1 \cdot
\be_\mu[M_n^q]^{1/q}\\
\leqslant &\ c \frac{q}{\max(p, \log(\epsilon+q))} \left(
\be_\mu[U_n^p]^{1/p} + 8^{1/p_0} e \be_\mu[M_n^q]^{1/q} + \max(1,
\frac{\log(\epsilon+p_0)}{p_0}) \be_\mu[M_n^q]^{1/q} \right).
\end{align*}

\noindent The second inequality in the theorem now follows.
\end{proof}

\subsection*{Acknowledgments}
We thank David Montague and Doug Sparks for providing detailed feedback
on an early draft of the paper, which improved the exposition.
We also thank the referees for a careful reading of the paper.



\end{document}